\documentclass{ifacconf}

\usepackage
{
    graphicx,
    amssymb,
    amsmath,
    dsfont, 
    xcolor,
    mathtools,
    enumitem,
    tikz,
    natbib,
    cuted,
}

\usepackage[notrig]{physics}

\usepackage[utf8]{inputenc}

\usetikzlibrary{arrows, shapes, automata, positioning}

\graphicspath{{pics/}}

\DeclareMathAlphabet{\mathpzc}{OT1}{pzc}{m}{it}

\newcommand{\subfiguretitle}[1]{{\scriptsize{#1}} \\}
\newcommand{\R}{\mathbb{R}}                                      
\newcommand{\ts}{\hspace*{0.1em}}                                
\newcommand{\mc}[2][]{\mathpzc{#2}{\smash[t]{\mathstrut}}_{#1}}  
\providecommand{\norm}[1]{\left\lVert #1 \right\rVert}           

\newcommand\xqed[1]{\leavevmode\unskip\penalty9999 \hbox{}\nobreak\hfill \quad\hbox{#1}}
\newcommand{\exampleSymbol}{\xqed{$\triangle$}}

\DeclareMathOperator{\diag}{diag}

\newtheorem{theorem}{Theorem}[section]

\newtheorem{lemma}[theorem]{Lemma}
\newtheorem{proposition}[theorem]{Proposition}
\newtheorem{definition}[theorem]{Definition}
\theoremstyle{definition}
\newtheorem{example}[theorem]{Example}
\newtheorem{remark}[theorem]{Remark}

\newenvironment{proof}{\textit{Proof:}}{\hfill$\square$}

\makeatletter
\renewcommand*\env@matrix[1][*\c@MaxMatrixCols c]{%
  \hskip -\arraycolsep
  \let\@ifnextchar\new@ifnextchar
  \array{#1}}
\makeatother

\mathtoolsset{centercolon} 

\allowdisplaybreaks

\begin{document}
\begin{frontmatter}

\title{Data-driven network analysis \\ using local delay embeddings}

\author[1]{Stefan Klus}
\author[2]{Hongyu Zhu}

\address[1]{Heriot--Watt University, Edinburgh, UK \\ (e-mail: s.klus@hw.ac.uk).}
\address[2]{RTX Technology Research Center, East Hartford, CT, USA \\ (e-mail: HongyuAlice.Zhu@rtx.com)}

\begin{abstract} 
Data-driven methods for the identification of the governing equations of dynamical systems or the computation of reduced surrogate models play an increasingly important role in many application areas such as physics, chemistry, biology, and engineering. Given only measurement or observation data, data-driven modeling techniques allow us to gain important insights into the characteristic properties of a system, without requiring detailed mechanistic models. However, most methods assume that we have access to the full state of the system, which might be too restrictive. We show that it is possible to learn certain global dynamical features from local observations using delay embedding techniques, provided that the system satisfies a localizability condition---a property that is closely related to the observability and controllability of linear time-invariant systems.
\end{abstract}


\end{frontmatter}

\section{Introduction}

Time-delay embeddings, theoretically justified by Takens' embedding theorem in the context of attractor reconstruction, are used in many application areas to identify more powerful models by including history data \citep{Pan20}. The required number of delays, however, is typically not known in advance and depends strongly on the system. The question we will address is how much information about global properties of the system is encoded in local observations, assuming we cannot access the full state. This is, for instance, the case when we have to deal with networked systems. The dynamical system could represent a decentralized network of agents. Each agent can observe and update its own state---by requesting information from adjacent agents---, but does not have access to the global state of the system or the network structure itself. The analysis of spectral properties of graphs using sparse local measurements, called \emph{spectral network identification}, is discussed in \citet{Mauroy17}. Local delay embeddings have also been used for analyzing fluid flows based on local PIV measurements in \citet{YZZWL21}. Furthermore, a decentralized spectral clustering method that relies on propagating waves through the graph and retrieving the clustering information with the aid of a local dynamic mode decomposition was proposed in \citet{ZKS22}. In this paper, we will analyze under which conditions we can expect the localized system to contain information about characteristic global properties such as inherent timescales or associated modes and focus mostly on linear systems. While linear systems have already been studied in detail, see, e.g., \citet{Willems86}, we show how these approaches can be extended to nonlinear systems by embedding the dynamics into a higher-dimensional feature space, provided that finite-dimensional Koopman-invariant subspaces exist. Moreover, we utilize recently developed data-driven methods for approximating the Koopman operator to show that spectral properties can be estimated from local trajectory data. The notion of \emph{localizability} allows us to assess whether or not it is possible to locally predict the global evolution of the dynamical system and also whether we can trust decentralized spectral clustering algorithms. All the results will be illustrated with the aid of guiding examples.

\section{Koopman operator theory}
\label{sec:Koopman}

Let $ \mathbb{X} \subset \mathbb{R}^n $ denote the state space and $ \Phi \colon \mathbb{X} \to \mathbb{X} $ a function. For a discrete deterministic dynamical system
\begin{equation*}
    x^{(k+1)} = \Phi(x^{(k)}),
\end{equation*}
the Koopman operator $ \mathcal{K} \colon L^\infty(\mathbb{X}) \to L^\infty(\mathbb{X}) $ applied to an observable $ f \in L^\infty(\mathbb{X}) $ is defined by
\begin{equation*}
    [\mathcal{K} f](x) = f(\Phi(x)),
\end{equation*}
see \citet{Ko31, LaMa94, RMBSH09} for more details. The Koopman operator and its eigenvalues, eigenfunctions, and modes are typically estimated from trajectory data. Given one long trajectory $ \{ x^{(0)}, x^{(1)}, \dots, x^{(m)} \} $ of the dynamical system $ \Phi $, we define the data matrices $ X, Y \in \R^{n \times m} $, with
\begin{align*}
    X &=
    \begin{bmatrix}
        x^{(0)} & x^{(1)} & \dots & x^{(m-1)}
    \end{bmatrix},\\
    Y &=
    \begin{bmatrix}
        x^{(1)} & x^{(2)} & \dots & x^{(m)}
    \end{bmatrix}.
\end{align*}
One of the simplest and most frequently used methods to estimate the Koopman operator is \emph{dynamic mode decomposition} (DMD) \citep{Schmid10, TRLBK14}. The underlying idea is to approximate the dynamics by a linear system of the form $ x^{(k+1)} = C \ts x^{(k)} $, where the matrix $ C := Y \ts X^+ \in \R^{n \times n} $ minimizes the Frobenius norm error $ \norm{C \ts X - Y}_F $. While this can be regarded as an approximation of the Koopman operator using linear basis functions, many nonlinear extensions using either explicitly given dictionaries \citep{WKR15, KKS16}, kernels \citep{WRK15, KSM19}, or neural networks \citep{LDBK17, MPWN18} exist. Instead of applying nonlinear transformations to the data, another popular approach to obtain more accurate models is to utilize \emph{delay embeddings} \citep{AM17, Pan20}. Given one long trajectory, the data-driven methods are then not directly applied to the matrices $ X $ and $ Y $, but to augmented data matrices $ X_s, Y_s \in \R^{(s \ts n) \times (m-s+1)} $ of the form
\begin{align*}
    X_s =
    \begin{bmatrix}
        x^{(0)} & x^{(1)} & \dots & x^{(m-s)} \\
        x^{(1)} & x^{(2)} & \dots & x^{(m-s+1)} \\
        \vdots & \vdots & \ddots & \vdots \\
        x^{(s-1)} & x^{(s)} & \dots & x^{(m-1)}
    \end{bmatrix}, \\
    Y_s =
    \begin{bmatrix}
        x^{(1)} & x^{(2)} & \dots & x^{(m-s+1)} \\
        x^{(2)} & x^{(3)} & \dots & x^{(m-s+2)} \\
        \vdots & \vdots & \ddots & \vdots \\
        x^{(s)} & x^{(s+1)} & \dots & x^{(m)}
    \end{bmatrix},
\end{align*}
where $ s $ determines the number of previous time points used for prediction. The goal now is to estimate a matrix $ C_s \in \R^{(s \ts n) \times (s \ts n)} $ from simulation data such that $ \norm{C_s X_s - Y_s}_F $ is minimized. Since $ s-1 $ components of the augmented vectors contained in $ X_s $ and $ Y_s $ are identical, we impose the following structure, which simplifies the regression problem:
\begin{equation*}
    C_s =
    \begin{bmatrix}
        & I \\
        & & I \\
        & & & \ddots \\
        & & & & I \\
        W_0 & W_1 & W_2 & \dots & W_{s-1}
    \end{bmatrix}.
\end{equation*}
We then obtain a model of the form
\begin{equation*}
    x^{(k+s)} = W_{s-1} \ts x^{(k+s-1)} + W_{s-2} \ts x^{(k+s-2)} + \dots + W_0 \ts x^{(k)}.
\end{equation*}
Relationships with similar approaches such as autoregressive models are described in \citet{Pan20}.

\section{Local delay embeddings}

We will now consider \emph{local delay embeddings} assuming that we cannot observe the full state of the system.

\subsection{Linear dynamical systems}

Let the linear dynamical system be given by
\begin{equation} \label{eq:Discrete linear system}
    x^{(k+1)} = A \ts x^{(k)},
\end{equation}
where $ A \in \R^{n \times n} $ is a matrix. By interpreting the matrix $ A $ as a weighted adjacency matrix, we can view this as a dynamical system defined on a directed graph.

\begin{definition}[Dependency graph]
Given a discrete dynamical system of the form \eqref{eq:Discrete linear system}, the associated \emph{dependency graph} is defined by $ \mc{G}(A) = (\mc{X}, \mc{E}) $, with vertices $ \mc{X} = \{ \mc[1]{x}, \dots, \mc[n]{x} \} $ and edges $ \mc{E} = \{ (\mc[j]{x}, \mc[i]{x}) \mid a_{ij} \ne 0 \} $.
\end{definition}

The dependency graph describes which variables $ x_j $ are required for the update of the variable $ x_i $. For linear systems, the adjacency matrix of the dependency graph is simply determined by the nonzero entries of $ A^\top $. The dependency graph can, however, be easily extended to nonlinear dynamical systems using the structure of the Jacobian of the system.

\subsection{Localizability}

Our goal is to identify global properties of the dynamical system using only local information. To this end, we write the matrix $ A $ as
\begin{equation*}
    A =
    \begin{bmatrix}
        a_{11} & a_{12}^\top \\
        a_{21} & A_{22}
    \end{bmatrix},
\end{equation*}
with $ a_{11} \in \R $, $ a_{12}, a_{21} \in \R^{n-1} $, and $ A_{22} \in \R^{(n-1) \times (n-1)} $. Similarly, we split $ x^{(k)} $ into
\begin{equation*}
    x^{(k)} =
    \begin{bmatrix}
        u^{(k)} \\
        v^{(k)}
    \end{bmatrix},
\end{equation*}
where $ u^{(k)} \in \R $ and $ v^{(k)} \in \R^{n-1} $. Assume now that we can only observe the evolution of $ u^{(k)} $, but do not have access to $ v^{(k)} $. That is, without loss of generality, we consider only the dynamics pertaining to the vertex $ \mc[1]{x} $ of the graph.

\begin{definition}[Localizability]
Let $ R \in \R^{(n-1) \times (n-1)} $ be defined by
\begin{equation*}
    R =
    \begin{bmatrix}
        a_{12}^\top \\
        a_{12}^\top \ts A_{22} \\
        \vdots \\
        a_{12}^\top \ts A_{22}^{n-2} \\
    \end{bmatrix}.
\end{equation*}
We call the discrete dynamical system \eqref{eq:Discrete linear system} \emph{localizable in vertex} $ \mc[1]{x} $ if $ \rank(R) = n - 1 $. Localizability in any other vertex can be defined analogously. We say that a system is \emph{localizable everywhere} if it is localizable in every vertex.\!\footnote{The term \emph{localizability} has also been used for network localization problems, where the goal is to determine the locations of nodes in a network, and should not be confused with our definition.}
\end{definition}

Note that localizability is closely related to the notions of \emph{observability} and \emph{controllability} used in control theory, see, e.g., \citet{Gajic96}. The meaning of the matrix $ R $ will become clear in Proposition~\ref{pro:Locally defined system}.

\begin{lemma}[Hautus test]
The system is localizable in $ \mc[1]{x} $ if and only if for each eigenvalue $ \lambda $ of $ A_{22} $ it holds that
\begin{equation*}
    \begin{bmatrix}
        \lambda \ts I - A_{22} \\
        a_{12}^\top
    \end{bmatrix}
    x
    = 0
\end{equation*}
implies that $ x = 0 $. That is, the matrix has rank $ n - 1 $ for all eigenvalues $ \lambda $.
\end{lemma}

\begin{proof}
This is simply an application of the Hautus test \citep{Hautus69} to the localizability setting.
\end{proof}

We seek to recover the dynamics using local information only. The following lemma shows that this is possible if the system is localizable.

\begin{proposition} \label{pro:Locally defined system}
Let the system be localizable in $ \mc[1]{x} $. Using a local delay embedding of length $ s = n $, i.e.,
\begin{align*}
    X_n =
    \begin{bmatrix}
        u^{(0)} & u^{(1)} & \dots & u^{(m-n)} \\
        u^{(1)} & u^{(2)} & \dots & u^{(m-n+1)} \\
        \vdots & \vdots & \ddots & \vdots \\
        u^{(n-1)} & u^{(n)} & \dots & u^{(m-1)}
    \end{bmatrix}, \\
    Y_n =
    \begin{bmatrix}
        u^{(1)} & u^{(2)} & \dots & u^{(m-n+1)} \\
        u^{(2)} & u^{(3)} & \dots & u^{(m-n+2)} \\
        \vdots & \vdots & \ddots & \vdots \\
        u^{(n)} & u^{(n+1)} & \dots & u^{(m)}
    \end{bmatrix},
\end{align*}
we can derive a linear system that correctly predicts the evolution of $ u^{(k)} $ without requiring~$ v^{(k)} $. That is, the localized system replicates the behavior of the global system.
\end{proposition}

The proof can be found in the appendix. If a system is localizable, this means that we can reproduce the dynamics of $ u^{(k)} $ using only local information and that the global dynamics can be reconstructed, provided that $ R $ and $ b^{(k)} $ (defined in the proof) are known. If the system is not localizable, we might still be able to reproduce the local dynamics, but lose information about the global behavior. We will estimate the matrices representing the localized dynamics from local trajectory data using DMD.

\begin{example} \label{ex:non-localizable systems}
Let us consider the case $ n = 3 $ and identify systems that are not localizable in $ \mc[1]{x} $. That is, we have a discrete dynamical system of the form $ x^{(k+1)} = A \ts x^{(k)} $, with
\begin{equation*}
    A =
    \left[\begin{array}{c|cc}
        a_{11} & a_{12} & a_{13} \\ \hline
        a_{21} & a_{22} & a_{23} \\
        a_{31} & a_{32} & a_{33}
    \end{array}\right].
\end{equation*}
Note that the matrix $ A $ is now defined element-wise and not block-wise. Since $ n = 3 $, we have $ v^{(k)} \in \R^2 $ and we want to express the dynamics in terms of $ u^{(k)} $ variables using delay embeddings. We obtain
\begin{equation*}
    R =
    \begin{bmatrix}
        a_{12} & a_{13} \\
        a_{12} \ts a_{22} + a_{13} \ts a_{32} & a_{12} \ts a_{23} + a_{13} \ts a_{33}
    \end{bmatrix}.
\end{equation*}
If $ \rank(R) < 2 $, the system is not localizable in $ \mc[1]{x} $. Examples of such systems and their corresponding dependency graphs---self-loops are omitted for the sake of clarity---are:
\begin{center}
    \begin{minipage}[t]{0.3\linewidth}
        \centering
        \resizebox{1\textwidth}{!}{%
        \begin{tikzpicture}[
                > = stealth, 
                semithick, 
            ]

            \tikzstyle{every state}=[
                draw = black,
                thick,
                fill = white,
                minimum size = 5mm
            ]

            \node[state] (v1) {1};
            \node[state] (v2) [right=2cm of v1] {2};
            \node[state] (v3) [below right=1cm and 0.85cm of v1] {3};

            \path[<->] (v1) edge node {} (v2);
            \path[->] (v1.-67) edge node {} (v3.157);
            \path[->] (v2.232) edge node {} (v3.38);
        \end{tikzpicture}} \\[1.5ex]
        \resizebox{0.9\textwidth}{!}{$
        A =
        \begin{bmatrix}[rrr]
            \tfrac{3}{5} &  -\tfrac{1}{2} & 0\; \\[1ex]
           -\tfrac{1}{2} &  -\tfrac{3}{5} & 0\; \\[1ex]
           -1 & \tfrac{1}{2} & -\tfrac{1}{2}\;
        \end{bmatrix}
        $}
    \end{minipage}
    \hspace*{1ex}
    \begin{minipage}[t]{0.3\linewidth}
        \centering
        \resizebox{1\textwidth}{!}{%
        \begin{tikzpicture}[
                > = stealth, 
                semithick 
            ]

            \tikzstyle{every state}=[
                draw = black,
                thick,
                fill = white,
                minimum size = 5mm
            ]

            \node[state] (v1) {1};
            \node[state] (v2) [right=2cm of v1] {2};
            \node[state] (v3) [below right=1cm and 0.85cm of v1] {3};

            \path[<->] (v1) edge node {} (v2);
            \path[<->] (v1.-60) edge node {} (v3.150);
        \end{tikzpicture}} \\[1.5ex]
        \resizebox{0.9\textwidth}{!}{$
        A =
        \begin{bmatrix}[rrr]
            \tfrac{1}{2} & -\tfrac{2}{5} &  \tfrac{2}{5}\; \\[1ex]
           -\tfrac{2}{5} & -\tfrac{1}{2} &  0\;            \\[1ex]
            \tfrac{2}{5} &  0            & -\tfrac{1}{2}\;
        \end{bmatrix}
        $}
    \end{minipage}
    \hspace*{1ex}
    \begin{minipage}[t]{0.3\linewidth}
        \centering
        \resizebox{1\textwidth}{!}{%
        \begin{tikzpicture}[
                > = stealth, 
                semithick 
            ]

            \tikzstyle{every state}=[
                draw = black,
                thick,
                fill = white,
                minimum size = 5mm
            ]

            \node[state] (v1) {1};
            \node[state] (v2) [right=2cm of v1] {2};
            \node[state] (v3) [below right=1cm and 0.85cm of v1] {3};

            \path[<->] (v1) edge node {} (v2);
            \path[<->] (v1.-60) edge node {} (v3.150);
            \path[<->] (v2.240) edge node {} (v3.30);
        \end{tikzpicture}} \\[1.5ex]
        \resizebox{0.9\textwidth}{!}{$
        A =
        \begin{bmatrix}[rrr]
            1 & 1 & 2\; \\[1ex]
           -1 & -\tfrac{1}{3} & -1\; \\[1ex]
           -1 & -\tfrac{5}{6} & -\tfrac{3}{2}\;
        \end{bmatrix}
        $}
    \end{minipage}
\end{center}
Information about $ x_3 $ does not enter the first two equations of the system on the left-hand side so that the second column of $ R $ is zero. Although the other two graphs are strongly connected or even fully connected, the corresponding systems are not localizable in $ \mc[1]{x} $. \exampleSymbol
\end{example}

\begin{remark}
While the non-localizability of the first system is purely due to the graph structure, perturbing, e.g., the entries $ a_{22} $ or $ a_{33} $ of the other two matrices will render the systems localizable in $ \mc[1]{x} $. We could now define \emph{structural localizability}, mirroring the definitions of \emph{structural observability} and \emph{structural controllability} \citep{Lin74}. Structurally localizable would then mean that there exists a system with the same structure that is localizable.
\end{remark}

Intuitively, it is clear that information from all the other vertices needs to reach each vertex for a system to be localizable everywhere.

\begin{lemma}
If the dependency graph is not strongly connected, then the system is not localizable in every vertex.
\end{lemma}

\begin{proof}
Assume that the underlying graph is not strongly connected. That is, $ A $ is reducible and we can find a permutation $ P $ such that
\begin{equation*}
    \widetilde{A} := P^\top \! A \ts P =
    \begin{bmatrix}
        B_{11} & 0 \\
        B_{21} & B_{22}
    \end{bmatrix},
\end{equation*}
with $ B_{11} \in \R^{p \times p} $, $ B_{22} \in \R^{(n-p) \times (n-p)} $, and $ 1 \le p \le \mbox{$n-1$} $, see \citet{Fie08}. If $p=1$, all entries of $ \widetilde{a}_{12} $ are zero and hence all entries of $R$ are zero. If $ p \ge 2 $, the last $ n-p $ entries of $ \widetilde{a}_{12} $ are zero and $ \widetilde{A}_{22} $ and hence also its powers are block-lower-triangular matrices, where the zero matrix is of size $ (p-1) \times (n-p) $. Multiplying $ \widetilde{a}_{12}^\top $ and $ \widetilde{A}_{22}^l $, with $ l = 0, \dots, n-2 $, results in row vectors containing $ n - p $ zeros so that $ n - p \ge 1 $ columns of the $ R $ matrix associated with $ \widetilde{A} $ are zero. Hence, $ R $ cannot have rank $ n - 1 $ and the system is not localizable everywhere.
\end{proof}

However, strongly connected systems are not necessarily localizable as illustrated in Example~\ref{ex:non-localizable systems}. Let us now assume that the system is localizable in $ \mc[1]{x} $. Setting $ s = n $, we want to find the matrix $ C_n \in \R^{n \times n} $ such that
\begin{equation*}
    \begin{bmatrix}
        u^{(k+1)} \\
        u^{(k+2)} \\
        \vdots \\
        u^{(k+n-1)} \\
        u^{(k+n)}
    \end{bmatrix}
    =
    \underbrace{
    \begin{bmatrix}
        & 1 \\
        & & 1 \\
        & & & \ddots \\
        & & & & 1 \\
        w_0 & w_1 & w_2 & \dots & w_{n-1}
    \end{bmatrix}}_{C_n}
    \begin{bmatrix}
        u^{(k)} \\
        u^{(k+1)} \\
        \vdots \\
        u^{(k+n-2)} \\
        u^{(k+n-1)}
    \end{bmatrix},
\end{equation*}
imposing the structure described in Section~\ref{sec:Koopman}. We have shown in Proposition~\ref{pro:Locally defined system} that such a system exists. Note that the delay embedding is only applied to the first vertex here. We simply estimate the matrix $ C_n $ from the augmented data matrices $ X_n, Y_n \in \R^{n \times (m-n+1)} $ containing only $ u^{(k)} $ trajectory data using DMD.

\subsection{Local estimation of eigenvalues}

\begin{figure*}
    \centering
    \begin{minipage}[t]{0.14\linewidth}
        \centering
        \subfiguretitle{(a)}
        \vspace*{0.6ex}
        \begin{tikzpicture}[
                > = stealth, 
                semithick 
            ]

            \tikzstyle{every state}=[
                draw = black,
                thick,
                fill = white,
                minimum size = 5mm
            ]

            \node[state] (v1) {1};
            \node[state] (v2) [below=0.6cm of v1] {2};
            \node[state] (v3) [below=0.6cm of v2] {3};
            \node[state] (v4) [right=1.2cm of v1]{4};
            \node[state] (v5) [below=0.6cm of v4] {5};
            \node[state] (v6) [below=0.6cm of v5] {6};

            \path[<->] (v1) edge node {} (v4);
            \path[<->] (v2) edge node {} (v5);
            \path[<->] (v3) edge node {} (v6);

            \path[->] (v4) edge node {} (v2);
            \path[->] (v5) edge node {} (v1);
            \path[->] (v5) edge node {} (v3);
            \path[->] (v6) edge node {} (v2);
        \end{tikzpicture}
    \end{minipage}
    \hspace*{3ex}
    \begin{minipage}[t]{0.2\linewidth}
        \centering
        \subfiguretitle{(b)}
        \vspace*{0.6ex}
        \includegraphics[width=\linewidth]{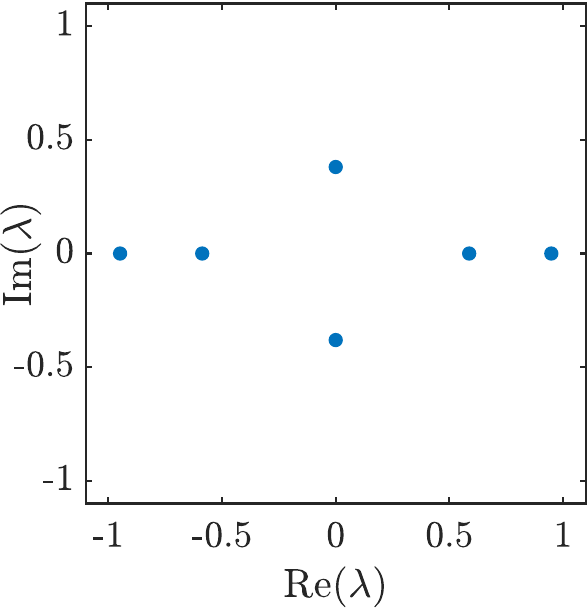}
    \end{minipage}
    \hspace*{3ex}
    \begin{minipage}[t]{0.259\linewidth}
        \centering
        \subfiguretitle{(c)}
        \includegraphics[width=\linewidth]{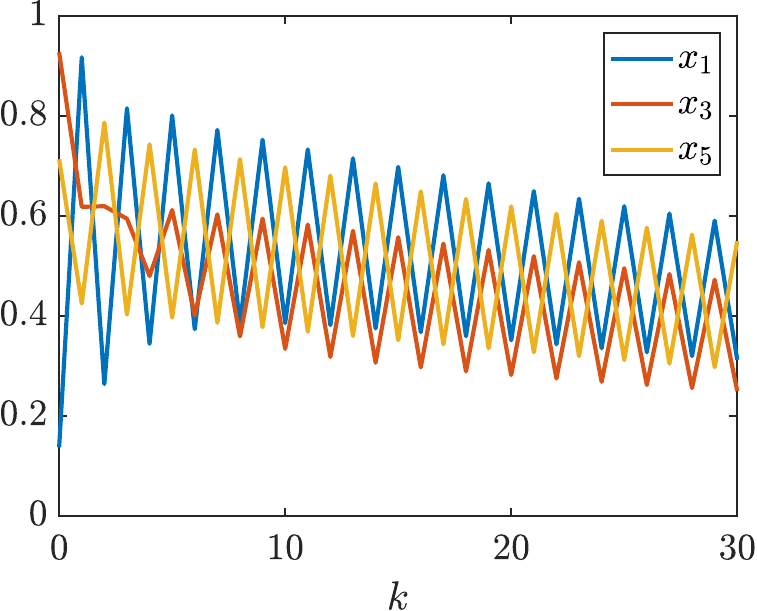}
    \end{minipage}
    \vspace*{-1ex}
    \caption{(a) Bipartite dependency graph of a system that is localizable everywhere. (b)~Eigenvalues of the graph. The spectrum is invariant under multiplication by $ -1 $. (c)~Trajectory data for the variables $ x_1 $, $ x_3 $, and $ x_5 $. Using any of these trajectories, we can determine that the graph is bipartite using only local information.}
    \label{fig:Bipartite graph}
\end{figure*}
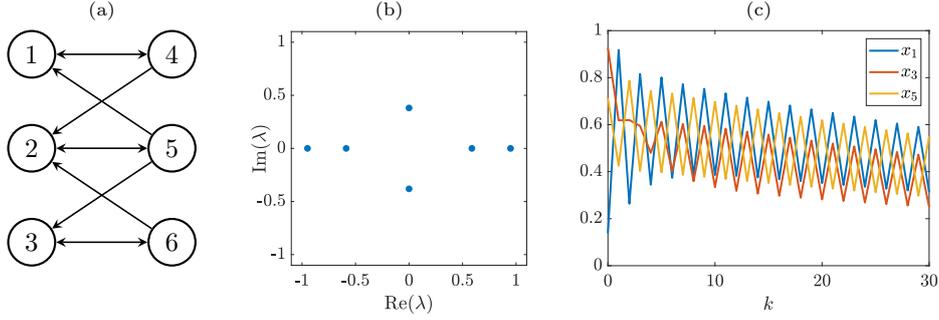

In what follows, we are mainly interested in spectral properties of the matrices $ A $ and $ C_n $.

\begin{lemma} \label{lem:Isospectral}
Let the system be localizable in vertex $ \mc[1]{x} $. If $ \lambda $ is an eigenvalue of $ A $, then $ \lambda $ is also an eigenvalue of $ C_n $ and the associated eigenvector is given by
\begin{equation*}
    \xi =
    \begin{bmatrix}
        u & \lambda \ts u & \dots & \lambda^{n-1} \ts u
    \end{bmatrix}^\top,
\end{equation*}
where $ u \ne 0 $.
\end{lemma}

\begin{proof}
Assume that
\begin{equation*}
    \begin{bmatrix}
        a_{11} & a_{12}^\top \\
        a_{21} & A_{22}
    \end{bmatrix}
    \begin{bmatrix}
        u \\
        v
    \end{bmatrix}
    =
    \lambda
    \begin{bmatrix}
        u \\
        v
    \end{bmatrix}.
\end{equation*}
Setting $ u^{(k)} = u $ and $ v^{(k)} = v $, we get $ u^{(k+1)} = \lambda \ts u $ and consequently $ u^{(k+s)} = \lambda^s \ts u $. We have shown in Proposition~\ref{pro:Locally defined system} that the localized system, given by $ C_n $, correctly describes the evolution of the variable $ u^{(k)} $. It follows that
\begin{equation*}
    C_n
    \begin{bmatrix}
        u \\
        \lambda \ts u \\
        \vdots \\
        \lambda^{n-1} \ts u
    \end{bmatrix}
    =
    \begin{bmatrix}
        \lambda \ts u \\
        \lambda^2 \ts u \\
        \vdots \\
        \lambda^n \ts u
    \end{bmatrix}
    =
    \lambda
    \begin{bmatrix}
        u \\
        \lambda \ts u \\
        \vdots \\
        \lambda^{n-1} \ts u
    \end{bmatrix},
\end{equation*}
which completes the proof.
\end{proof}

\begin{example}
Many important properties of graphs are encoded in its spectrum, see \citet{Brualdi10} for an overview. A directed graph, for instance, is bipartite if and only if its spectrum is invariant under multiplication by $ -1 $. Given local trajectory data, we can now determine if the underlying graph is bipartite using Lemma~\ref{lem:Isospectral}. This is illustrated in Figure~\ref{fig:Bipartite graph}. \exampleSymbol
\end{example}

\begin{lemma}
Let the system be localizable in vertex $ \mc[1]{x} $ and let the characteristic polynomial associated with $ A $ be given by
\begin{equation*}
    p_{\rule{0pt}{1.7ex}A}(\lambda) = \det(\lambda \ts I - A) = \lambda^n + \alpha_{n-1} \ts \lambda^{n-1} + \dots + \alpha_1 \ts \lambda + \alpha_0.
\end{equation*}
The entries $ w_i $ of $ C_n $ are then given by $ w_i = -\alpha_i $, where $ i = 0, \dots, n-1 $, i.e.,
\begin{equation*}
    C_n =
    \begin{bmatrix}
        & 1 \\
        & & 1 \\
        & & & \ddots \\
        & & & & 1 \\
        -\alpha_0 & -\alpha_1 & -\alpha_2 & \dots & -\alpha_{n-1}
    \end{bmatrix}.
\end{equation*}
\end{lemma}

\begin{proof}
The eigenvalues of $ A $ and $ C_n $ are identical as shown in Lemma~\ref{lem:Isospectral} and thus also their characteristic polynomials. The characteristic polynomial associated with the companion matrix $ C_n $ is
\begin{equation*}
    p_{\rule{0pt}{1.7ex}C_n}(\lambda) = \det(\lambda \ts I - C_n) = \lambda^n - w_{n-1} \ts \lambda^{n-1} - \dots - w_1 \ts \lambda - w_0,
\end{equation*}
see, e.g., \citet{GVL13}.
\end{proof}

Note that the last entry of the vector $ C_n \ts \xi $ in Lemma~\ref{lem:Isospectral} is then given by
\begin{equation*}
    -\sum_{l=0}^{n-1} \alpha_l \ts \lambda^l \ts u = \lambda^n \ts u
\end{equation*}
since $ p_{\rule{0pt}{1.7ex}A}(\lambda) = \lambda^n + \alpha_{n-1} \ts \lambda^{n-1} + \dots + \alpha_1 \ts \lambda + \alpha_0 = 0 $ if $ \lambda $ is an eigenvalue.

\begin{example}
For $ n = 2 $, we obtain
\begin{equation*}
    C_2 =
    \begin{bmatrix}
        0 & 1 \\
        -\det(A) & \tr(A)
    \end{bmatrix}
\end{equation*}
since
\begin{equation*}
    p_{\rule{0pt}{1.8ex}C_2}(\lambda) = \lambda^2 - (a_{11} + a_{22}) \lambda + a_{11} \ts a_{22} - a_{12} \ts a_{21}. \tag*{\exampleSymbol}
\end{equation*}
\end{example}

Indeed, for arbitrary $ n \in \mathbb{N} $, it holds that $ w_{n-1} = \tr(A) $ and $ w_0 = (-1)^{n+1} \det(A) $. That is, we can also compute the trace and determinant of the system locally. We could, of course, also directly compute these properties using the eigenvalues of $ C_n $, i.e., the trace is the sum and the determinant the product of the eigenvalues.

\subsection{Local estimation of eigenvectors}

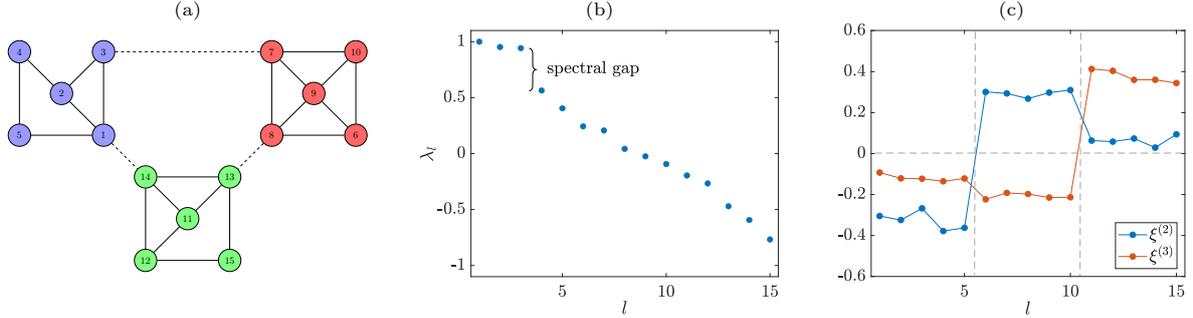
\begin{figure*}
    \centering
    \begin{minipage}[t]{0.26\linewidth}
        \centering
        \subfiguretitle{(a)}
        \vspace*{1.5ex}
        \resizebox{\textwidth}{!}{%
        \begin{tikzpicture}[
                > = stealth, 
                semithick 
            ]

            \tikzstyle{every state}=[
                draw = black,
                thick,
                fill = white,
                minimum size = 8mm
            ]

            \def\factor{3}
            \node[state, fill=blue!40]  at (\factor*-1.0, \factor*-0.5) (1) {1};
            \node[state, fill=blue!40]  at (\factor*-1.5, \factor* 0.0) (2) {2};
            \node[state, fill=blue!40]  at (\factor*-1.0, \factor* 0.5) (3) {3};
            \node[state, fill=blue!40]  at (\factor*-2.0, \factor* 0.5) (4) {4};
            \node[state, fill=blue!40]  at (\factor*-2.0, \factor*-0.5) (5) {5};
            \node[state, fill=red!60]   at (\factor* 2.0, \factor*-0.5) (6) {6};
            \node[state, fill=red!60]   at (\factor* 1.0, \factor* 0.5) (7) {7};
            \node[state, fill=red!60]   at (\factor* 1.0, \factor*-0.5) (8) {8};
            \node[state, fill=red!60]   at (\factor* 1.5, \factor* 0.0) (9) {9};
            \node[state, fill=red!60]   at (\factor* 2.0, \factor* 0.5) (10) {10};
            \node[state, fill=green!50] at (\factor* 0.0, \factor*-1.5) (11) {11};
            \node[state, fill=green!50] at (\factor*-0.5, \factor*-2.0) (12) {12};
            \node[state, fill=green!50] at (\factor* 0.5, \factor*-1.0) (13) {13};
            \node[state, fill=green!50] at (\factor*-0.5, \factor*-1.0) (14) {14};
            \node[state, fill=green!50] at (\factor* 0.5, \factor*-2.0) (15) {15};

            \path[-] (1) edge (2)
                     (1) edge (3)
                     (1) edge (5)
                     (2) edge (3)
                     (2) edge (4)
                     (4) edge (5)
                     (6) edge (8)
                     (6) edge (9)
                     (6) edge (10)
                     (7) edge (8)
                     (7) edge (9)
                     (7) edge (10)
                     (8) edge (9)
                     (9) edge (10)
                     (11) edge (12)
                     (11) edge (13)
                     (11) edge (14)
                     (12) edge (14)
                     (12) edge (15)
                     (13) edge (14)
                     (13) edge (15);
             \path[-, dashed] (1) edge (14)
                              (8) edge (13)
                              (3) edge (7);
        \end{tikzpicture}}
    \end{minipage}
    \hspace*{3ex}
    \begin{minipage}[t]{0.26\linewidth}
        \centering
        \subfiguretitle{(b)}
        \vspace*{0.55ex}
        \includegraphics[width=\linewidth]{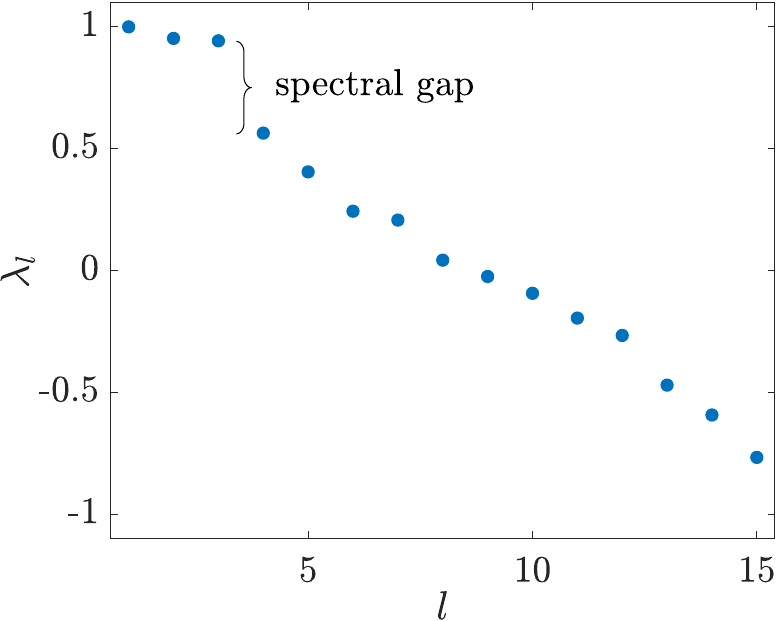}
    \end{minipage}
    \hspace*{3ex}
    \begin{minipage}[t]{0.26\linewidth}
        \centering
        \subfiguretitle{(c)}
        \includegraphics[width=0.97\linewidth]{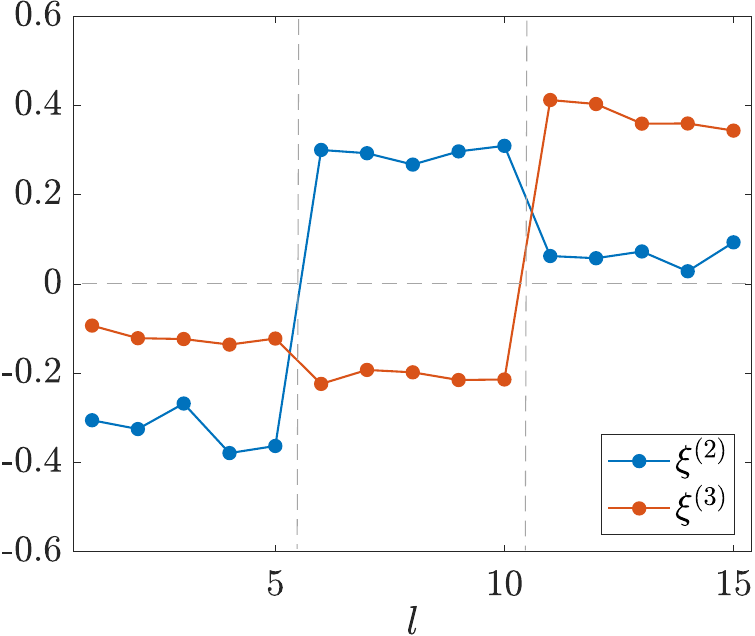}
    \end{minipage}
    \vspace*{-0.7ex}
    \caption{(a)~Decentralized spectral clustering of a randomly generated graph with three clusters. Dashed edges have a lower edge weight. Each vertex assigns itself to a cluster using the signs of the locally computed components of the second and third eigenvectors, where blue represents \texttt{--}, red \texttt{+-}, and green \texttt{++}.   (b)~Eigenvalues of the associated normalized graph Laplacian. (c) Second and third eigenvectors of the graph Laplacian.}
    \label{fig:Block graph}
\end{figure*}

The question is whether it is also possible to obtain information about properties of eigenvectors using only local trajectory data. For the full system, assuming $ A $ is diagonalizable, i.e., $ A = \Xi \ts \Lambda \ts \Xi^{-1} $, where $ \Xi = [\xi^{(1)}, \dots, \xi^{(n)}] $ contains the eigenvectors and $ \Lambda = \diag(\lambda_1, \dots, \lambda_n) $ is a diagonal matrix containing the eigenvalues, we have
\begin{align*}
    x^{(k)} &= A^k \ts x^{(0)} = \Xi \ts \Lambda^k \ts \underbrace{\Xi^{-1} x^{(0)}}_{:=z} = \Xi \ts \Lambda^k \ts z \\
        &= \Xi \ts \diag(z) \begin{bmatrix} \lambda_1^k \\ \vdots \\ \lambda_n^k \end{bmatrix} = \begin{bmatrix} z_1 \ts \xi^{(1)} & \dots & z_n \ts \xi^{(n)} \end{bmatrix} \begin{bmatrix} \lambda_1^k \\ \vdots \\ \lambda_n^k \end{bmatrix}.
\end{align*}
Considering only the first variable of $ x^{(k)} $, i.e., $ u^{(k)} $, this yields
\begin{equation*}
    \begin{bmatrix}
        u^{(0)} \, \dots \, u^{(m-1)}
    \end{bmatrix}
    =
    \begin{bmatrix}
        z_1 \ts \xi_1^{(1)} \, \dots \, z_n \ts \xi_1^{(n)}
    \end{bmatrix}
    \begin{bmatrix}
        1      & \lambda_1 & \dots  & \lambda_1^{m-1} \\
        1      & \lambda_2 & \dots  & \lambda_2^{m-1} \\
        \vdots & \vdots    & \ddots & \vdots          \\
        1      & \lambda_n & \dots  & \lambda_n^{m-1}
    \end{bmatrix}.
\end{equation*}
For $ m > n $, this is a regression problem, which can be solved locally since the trajectory data $ u^{(k)} $ is known and the eigenvalues $ \lambda_l $ can be computed in each vertex as shown above. That is, we can locally compute the first component of each eigenvector up to the multiplicative constants $ z_l $, which are determined by the initial condition $ x^{(0)} $. This can be applied in the same way to all the other vertices. A similar result---based on a spectral decomposition of the signal---was shown in \citet{ZKS22} for a wave equation-based decentralized clustering approach for undirected graphs. We can now also, for example, detect the number of clusters---by looking for spectral gaps---and compute the sign structure of dominant eigenvectors in order to assign vertices to clusters using only local information.

\begin{example}
We generate a simple graph using the stochastic block model and simulate the resulting dynamical system, where $ A $ is now the normalized graph Laplacian~$ L $ of the system shown in Figure~\ref{fig:Block graph}. Based only on local trajectory data, each vertex computes the eigenvalues---the spectral gap implies that the graph can be decomposed into three weakly coupled clusters---and its local eigenvector entries and assigns itself to one of the three clusters. \exampleSymbol
\end{example}

The decentralized clustering algorithm proposed in \citet{ZKS22} relies on a discretization of the wave equation $ u_{tt} = c^2 \Delta u $ on a graph. The corresponding linear system can be written as
\begin{equation*}
    x^{(k+1)} =
    \begin{bmatrix}
        2 \ts I - c^2 L ~~ & -I \\
        I & 0
    \end{bmatrix}
 \ts x^{(k)},
\end{equation*}
where $ c $ is the wave speed and $ L $ again the normalized graph Laplacian. If $ 0 < c < \sqrt{2} $, then the eigenvalues of the matrix all lie on the unit circle. The advantage of this formulation, which can be regarded as a special case of our framework, is that the eigenvalues, unlike in the above example, are not damped out.

\subsection{Extensions to nonlinear dynamical systems}

The Koopman operator allows us to represent nonlinear dynamical systems by an associated infinite-dimensional but linear operator. Provided that finite-dimensional Koopman-invariant subspaces exist, we can obtain a matrix representation of the system that can be reconstructed locally using time-delay embeddings.

\begin{example}
Let us consider a network of $ d = 4 $ discrete nonlinear dynamical systems, where the dynamics of the uncoupled subsystems are given by $ \phi_i \colon \R^2 \to \R^2 $, with
\begin{equation*}
    \phi_i(x_i) =
    \begin{bmatrix}
        \alpha_i \ts x_{i,1} + \beta_i (x_{i,2}^3 - x_{i,2}) \\
        \gamma_i \ts x_{i,2}
    \end{bmatrix}.
\end{equation*}
The global dynamics are then defined by
\begin{equation*}
    \Phi_i(x_1, \dots, x_d) = \phi_i(x_i) + \varepsilon \sum_{j=1}^d S_{ij} \begin{bmatrix} x_{j,1} \\ 0 \end{bmatrix},
\end{equation*}
for $ i = 1, \dots, d $. We choose the coupling structure shown in Figure~\ref{fig:Coupled cell system}\ts(a), which is given by the matrix
\begin{equation*}
    S =
    \begin{bmatrix}
        0 & 1 & 0 & 0 \\
        0 & 0 & 1 & 1 \\
        0 & 0 & 0 & 1 \\
        1 & 0 & 0 & 0
    \end{bmatrix}.
\end{equation*}
By introducing a new variable $ x_{i,3} := x_{i,2}^3 $, the subsystem dynamics become linear. This is equivalent to applying EDMD using a dictionary containing $ \{ x_{i,1}, x_{i,2}, x_{i,2}^3 \} $. We obtain the augmented system $ \phi_{i}(\bar{x}_i) = A_i \ts \bar{x}_i $, with $ \bar{x}_i = [x_{i,1}, x_{i,2}, x_{i,3}]^\top $ and
\begin{equation*}
    A_i =
    \begin{bmatrix}
        \alpha_i & -\beta_i & \beta_i \\
        0 & \gamma_i & 0 \\
        0 & 0 & \gamma_i^3
    \end{bmatrix}.
\end{equation*}
Due to the linear coupling, we can then write the global dynamics as a twelve-dimensional linear system, whose structure is shown in Figure~\ref{fig:Coupled cell system}\ts(b). For the numerical simulation, we set $ \varepsilon = 0.1 $ and randomly choose parameters $ \alpha_i \in [-1, 0] $, $ \beta_i \in [1, 2] $, and $ \gamma_i \in [-1, 0] $. The resulting system is localizable. We can thus rewrite the dynamics so that only local information is used to evolve the system dynamics. The trajectory for $ x_{1, 1} $ is shown in Figure~\ref{fig:Coupled cell system}\ts(c), where we compare the solutions obtained by simulating the original nonlinear system and the locally determined linear system. \exampleSymbol

\begin{figure*}
    \centering
    \begin{minipage}[t]{0.17\linewidth}
        \centering
        \subfiguretitle{(a)}
        \vspace*{2ex}
        \resizebox{1\textwidth}{!}{%
        \begin{tikzpicture}[
                > = stealth, 
                semithick 
            ]

            \tikzstyle{every state}=[
                shape = rectangle,
                draw = black,
                thick,
                fill = white,
                minimum size = 6mm
            ]

            \node[state] (v1) {1};
            \node[state] (v2) [right=1.8cm of v1] {2};
            \node[state] (v4) [below=1.8cm of v1] {4};
            \node[state] (v3) [right=1.8cm of v4] {3};

            \path[->] (v2) edge node {} (v1);
            \path[->] (v3) edge node {} (v2);
            \path[->] (v4) edge node {} (v2);
            \path[->] (v4) edge node {} (v3);
            \path[->] (v1) edge node {} (v4);
        \end{tikzpicture}}
    \end{minipage}
    \hspace*{3ex}
    \begin{minipage}[t]{0.205\linewidth}
        \centering
        \subfiguretitle{(b)}
        \vspace*{0.6ex}
        \includegraphics[width=\linewidth]{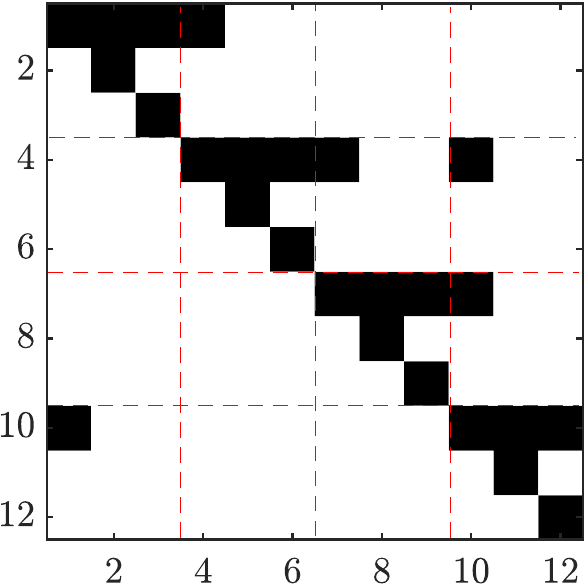}
    \end{minipage}
    \hspace*{3ex}
    \begin{minipage}[t]{0.30\linewidth}
        \centering
        \subfiguretitle{(c)}
        \includegraphics[width=\linewidth]{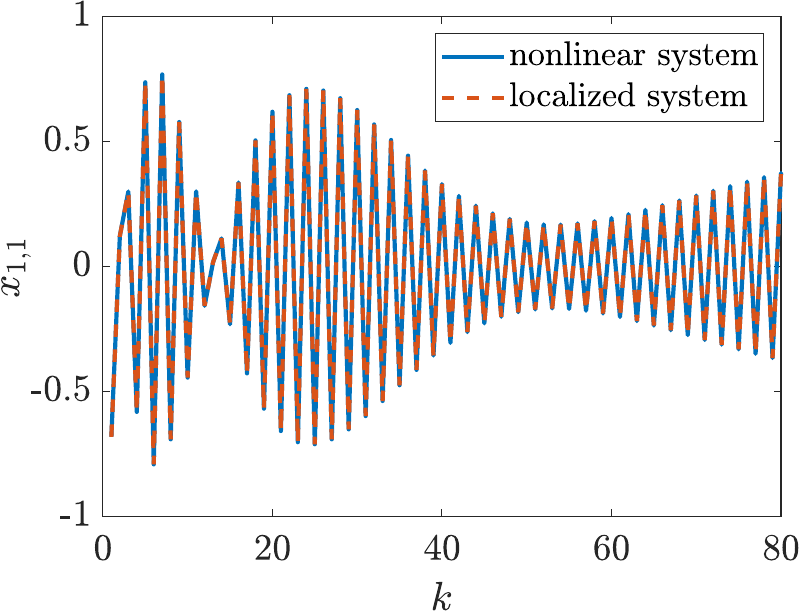}
    \end{minipage}
    \vspace*{-1ex}
    \caption{(a) Coupling structure $ S $ of the system. Each vertex represents a two-dimensional nonlinear dynamical system. (b) Structure of the augmented linear system. (c)~Evolution of $ x_{1, 1} $ computed by evaluating the original nonlinear system and the corresponding localized system learned from trajectory data.}
    \label{fig:Coupled cell system}
\end{figure*}
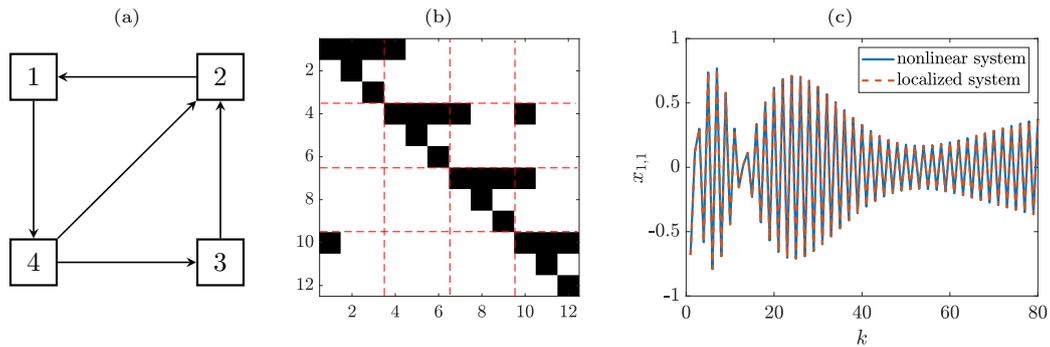

\end{example}

The example illustrates that by first embedding the dynamics into a higher-dimensional feature space and then using delay embedding techniques, we can derive a new dynamical system that depends only on local information, but still captures the global dynamics. Alternatively, we could directly apply delay embedding techniques to the nonlinear system, but would then not a priori know the number of delays required to faithfully reconstruct the local dynamics. Finding Koopman-invariant subspaces, provided they exist, is in general a challenging problem.

\section{Conclusion}

We have shown how local delay embeddings, combined with data-driven methods to learn the localized dynamics, can be used to detect global properties of linear and simple nonlinear dynamical systems. Although we cannot expect to be able to predict the behavior of more complex nonlinear systems over long time periods, it might still be possible to make short-term predictions that can be used within a model-predictive control framework. While we considered only discrete dynamical systems, the proposed approach can be easily extended to linear ordinary differential equations $ \dot{x} = B \ts x $ and also simple nonlinear systems by embedding the dynamics into a higher-dimensional feature space. If we, for instance, use a Runge--Kutta integrator to generate time-series data, this results in a discrete dynamical system, with $ A \approx \exp(h \ts B) $, where $ h $ is the step size of the integrator. Note, however, that higher-order integrators and implicit integrators might destroy the sparsity of the system. An interesting application would be the analysis of consensus networks. Is it possible to predict and then also control, e.g., the consensus value locally by adapting the coupling strengths? There are currently several limitations that might hamper the applicability of the proposed methods: Since the number of delays required to guarantee that the system is localizable depends on the number of vertices, applying this approach to large-scale networks would require infeasibly long trajectories. Furthermore, the resulting matrices might be ill-conditioned. Is it possible to lower the number of time-delay embeddings depending on the diameter of the graph to ensure that information from each vertex has reached every other vertex? Furthermore, it is not clear how sensitive the estimation of the spectral properties is with respect to noisy observations or measurement errors. These questions will be addressed in future work.

\bibliography{LDE}

\newpage

\onecolumn

\section*{Appendix}

\appendix

\textit{Proof of Proposition \ref{pro:Locally defined system}:}
We have
\begin{align*}
    u^{(k+1)} &= a_{11} \ts u^{(k)} + a_{12}^\top \ts v^{(k)}, \\
    v^{(k+1)} &= a_{21} \ts u^{(k)} + A_{22} \ts v^{(k)},
\end{align*}
so that
\begin{align*}
    u^{(k+2)} &= a_{11} \ts u^{(k+1)} + a_{12}^\top \ts v^{(k+1)} \\
              &= a_{11} \ts u^{(k+1)} + a_{12}^\top \ts a_{21} \ts u^{(k)} + a_{12}^\top A_{22} \ts v^{(k)}.
\end{align*}
Continuing this process, we obtain
\begin{align*}
    u^{(k+s)} &= a_{11} \ts u^{(k+s-1)} + \sum_{l=0}^{s-2} a_{12}^\top \ts A_{22}^l \ts a_{21} \ts u^{(k+s-2-l)} \\ &~~ + a_{12}^\top \ts A_{22}^{s-1} v^{(k)}.
\end{align*}
We write the first $ n - 1 $ equations as \\
\begin{equation*}
    \underbrace{
    \begin{bmatrix}
        a_{12}^\top \\[0.5ex]
        a_{12}^\top \ts A_{22} \\[1ex]
        \vdots \\[1ex]
        a_{12}^\top \ts A_{22}^{n-2} \\[2ex]
    \end{bmatrix}}_{=R}
    v^{(k)}
    =
    \underbrace{
    \begin{bmatrix}
        u^{(k+1)} - a_{11} \ts u^{(k)} \\
        u^{(k+2)} - a_{11} \ts u^{(k+1)} - a_{12}^\top \ts a_{21} \ts u^{(k)} \\
        \vdots \\
        \displaystyle u^{(k+n-1)} - a_{11} \ts u^{(k+n-2)} - \sum_{l=0}^{n-3} a_{12}^\top \ts A_{22}^l \ts a_{21} \ts u^{(k+n-3-l)}
    \end{bmatrix}}_{=:b^{(k)}}.
\end{equation*}
In order to obtain $ v^{(k)} $ using only $ u^{(k)} $ time-series data, we have to solve the system of linear equations. Since $ \rank(R) = n - 1 $ by assumption, the solution is unique and $ v^{(k)} $ can be written in terms of the variables $ u^{(k)}, \dots, u^{(k+n-1)} $. Plugging this expression for $ v^{(k)} $ into the formula for $ u^{(k+n)} $, we obtain a new discrete linear dynamical system that does not depend on $ v^{(k)} $. That is, the dynamics now depend only on $ u^{(k)}, \dots, u^{(k+n-1)} $.

\end{document}